\documentclass[12pt]{article}
\usepackage{amssymb,amsmath,enumerate,latexsym,epsf,graphicx}

\newfont{\bb}{msbm10 at 12pt}
\newfont{\bbp}{msbm10 at 9pt}

\def\r{\hbox{\bb R}}

\def\h{\hbox{\bb H}}
\def\tm{\hbox{\bbp M}^2}
\def\mr{\hbox{\bb M}^2\times\hbox{\bb R}}

\def\man{\mathcal{M}}
\def\hm{\mathcal{M}(\kappa , \tau)}

\def\m{\hbox{\bb M}^2}

\def\s{\hbox{\bb S}}

\newcommand{\camb}{\overline{\nabla}}

\newcommand{\norm}[1]{\left\Vert #1 \right\Vert}

\newcommand{\set}[1]{\left\{#1\right\}}
\newcommand{\meta}[2]{\langle #1,#2 \rangle }

\newcommand{\ext}{\wedge}
\newcommand{\campo}{\mathfrak{X}}

\newcommand{\bi}{\partial _{\infty}}

\begin{document}

\newtheorem{theorem}{Theorem}[section]
\newtheorem{proposition}[theorem]{Proposition}
\newtheorem{definition}[theorem]{Definition}
\newtheorem{lemma}[theorem]{Lemma}
\newtheorem{corollary}[theorem]{Corollary}
\newtheorem{remark}[theorem]{Remark}
\newtheorem{claim}[theorem]{Claim}

\newenvironment{proof}{\smallskip\noindent{\it Proof.}\hskip \labelsep}
                        {\hfill\penalty10000\raisebox{-.09em}{$\Box$}\par\medskip}

\title{Locally convex surfaces immersed\\[4mm] in a Killing submersion}

\author{Jos\'{e} M. Espinar\thanks{The author is partially
supported by Spanish MEC-FEDER Grant MTM2007-65249, and Regional J. Andaluc\'{i}a
Grants P06-FQM-01642 and FQM325}, In$\hat{\rm{e}}$s S. de Oliveira$\mbox{}^{\dag}$}
\date{}

\maketitle

\vspace{.4cm}

\noindent $\mbox{}^{\ast}$ Departamento de Geometr\'{i}a y Topolog\'{i}a,
Universidad de Granada, 18071 Granada, Spain; e-mail: jespinar@ugr.es\vspace{0.2cm}

\noindent $\mbox{}^\dag$ Pontif\'{i}cia Universidade Cat\'{o}lica do Rio de Janeiro,
Departamento de Matem\'{a}tica, Rio de janeiro ,Brasil; e-mail:
in.math@hotmail.com\vspace{0.2cm}


\vspace{.3cm}

\begin{abstract}
We generalize Hadamard-Stoker-Currier Theorems for surfaces immersed in a Killing
submersion over a strictly Hadamard surface whose fibers are the trajectories of a
unit Killing field. We prove that every complete surface whose principal curvatures
are greater than a certain function (depending on the ambient manifold) at each
point, must be properly embedded, homeomorphic to the sphere or to the plane and, in
the latter case, we study the behavior of the end.
\end{abstract}

\section{Introduction}

J. Stoker \cite{S} generalized the result of J.Hadamard \cite{H} that a compact
strictly locally convex surface in the Euclidean 3-space $\mathbb{R}^{3}$ is
embedded and homeomorphic to the sphere. Later on, J. Stoker showed that a complete
strictly locally convex immersed surface in $\mathbb{R}^{3}$ must be embedded and
homeomorphic to the sphere if it is closed or to the plane if it is open. In the
latter case, the surface is a graph over a planar domain. This result is known
currently as the Hadamard-Stoker Theorem.

Also M. Do Carmo and F.Warner \cite{CW} extended Hadamard's Theorem to the
hyperbolic 3-space $\mathbb{H}^{3}$ assuming the surface is compact and has positive
extrinsic curvature. The complete case in $\mathbb{H}^{3}$ was treated by
R.J.Currier in \cite{C}. Currier's Theorem says that a complete immersed surface in
$\mathbb{H}^{3}$ whose principal curvatures are greater than or equal to one, is
embedded and homeomorphic to the sphere if it is closed or to the plane if it is
open.

Recently, J. Espinar, J. G\'{a}lvez and H. Rosenberg \cite{EGR} extended the
Hadamard-Stoker theorem for immersed surfaces in $\mathbb{H}^{2} \times \mathbb{R}$
assuming that such a surface is connected, complete and whose extrinsic curvature is
positive. More precisely they showed that if $\Sigma$ is a complete connected
immersed surface in $\mathbb{H}^{2}\times \mathbb{R}$ with positive extrinsic
curvature, $K_{e} > 0$, then $\Sigma$ is properly embedded. Moreover, $\Sigma$ is
homeomorphic to $\mathbb{S}^{2}$ if it is closed or to $\mathbb{R}^{2}$ if it is
open. In the latter case, $\Sigma$ is a graph over a convex domain of
$\mathbb{H}^{2} \times \{0\}$ or $\Sigma$ has a simple end (we will make explicit
the definition of a simple end in Section \ref{HSTheorem}).

We work in this paper on immersed surfaces in Riemannian $3-$manifold which fiber
over a Riemmanian surface and whose fibers are the trajectories of a unit Killing
vector field. The study of immersed surfaces in such a manifold is a topic of
increasing interest (see \cite{RST} or \cite{LR} and references therein). In
particular, they include the metric product spaces $\mr$ for any Riemannian surface
$\m$, the Heisenberg spaces, the Berger spheres or $\widetilde{{\rm PSL}(2, \r)}$.

As far as we know, locally convex surfaces in $\widetilde{{\rm PSL}(2, \r)}$ have
not yet been studied. We prove a Hadamard-Stoker-Currier type theorem in these
spaces (more in the sense of Currier Theorem, i.e., giving conditions on the
principal curvatures of the surfaces instead of the extrinsic curvature). Actually,
basic problems on locally convex surface remain open, for example, it is not yet
known the classification of complete surfaces with positive extrinsic curvature in
$\widetilde{{\rm PSL}(2,\r)}$. Indeed, we do not know the parametrization of the
rotational spheres with constant extrinsic curvature in these spaces.

We start by establishing the notation and preliminaries results in Section
\ref{preliminaries}. Section \ref{preliminaries} is divided in three parts: the
first one focused on Hadamard surfaces, where we set up the basic notation and some
Lemmas on folitations by geodesics. Afterwards, we study Riemannian submersions over
Riemannian surfaces whose fibers are the trajectories of a unit Killing vector field
$\xi$. We parametrize such manifolds by two functions $\kappa$ and $\tau$, where
$\kappa $ is the curvature of the base. $\tau$ depends on $\xi$ (see Proposition
\ref{Prop:tau}). In the last part of Section \ref{preliminaries}, we study the
geometry of vertical cylinders and we set up the necessary concepts we use later on.

Section \ref{HSTheorem} is devoted to our main results, specifically

\begin{quote}
{\bf Theorem \ref{Theo:HS}} {\it Let $ \Sigma \subset \hm$ be a complete connected
immersed surface so that $k_{i}(p) > |\tau (p)|$ for all $p\in \Sigma$, where $\hm$
is a strict Hadamard-Killing submersion. Then $\Sigma$ is properly embedded.
Moreover, $\Sigma$ is homeomorphic to $\mathbb{S}^{2}$ or to $\mathbb{R}^{2}$. In
the later case, $\Sigma$ has a simple end (see Definition \ref{Def:simpleend}) or
$\Sigma$ is a Killing graph over a convex domain of $\m$.}
\end{quote}

The above Theorem \ref{Theo:HS} generalizes \cite[Theorem 3.1]{EGR}. Moreover, it
can be applied to surfaces in $\widetilde{{\rm PSL}(2,\r)}$.

\section{Preliminaries}\label{preliminaries}

\subsection{On Hadamard surfaces}

Here we will remind some definitions and results that we will need later. For more
details on Hadamard manifolds with non positive sectional curvature see \cite{E}.

Let $\m $ be a Hadamard surface, that is, $\m$ is a complete, simply connected
surface with Gaussian curvature $\kappa \leq 0$.

It is well known that given two points $p,q \in \m$, there exists an unique geodesic
$\gamma_{pq}$ joining $p$ and $q$. We say that two geodesics $\gamma, \beta$ in $\m$
are asymptotic if there exists a constant $C>0$ such that $d(\gamma (t), \beta (t))$
$\leq$ $C$ for all $t$ $>$ 0. To be asymptotic is an equivalence relation on the
oriented unit speed geodesics or on the set of unit vectors of $\m$. We will denote
by $\gamma(+\infty)$ and $\gamma(-\infty)$ the equivalence classes of the geodesics
$t \rightarrow \gamma(t)$ and $t \rightarrow \gamma(-t)$ respectively. Moreover, a
equivalence class is called point at infinity. $\m(\infty)$ denotes the set of all
points at infinity for $\m$ and $\m _{*} = \m\cup\m(\infty)$.

The set $\m _{*} = \m\cup \m(\infty)$ admits a natural topology, called the cone
topology, which makes $\m _{*}$ homeomorphic to the closed $2-$disk in $\r ^{2}$.


Let $p,q,r \in \m $ so that $q$ and $r$ are distinct from $p$. Then $\measuredangle
_p(q,r)$ denotes the angle at $p$ subtended by $q$ and $r$, that is, $\measuredangle
_p(q,r)$ is defined to be the angle between $\gamma'_{pq}(0)$ and $\gamma'_{pr}(0)$,
where $\gamma_{pq}$ and $\gamma_{pr}$ are the geodesics joining $p$ to $q$ and $p$
to $r$ respectively. Now, we recall two important results on geodesic triangles.

\begin{itemize}
\item \textbf{Law of cosines:} Let $p,q$ and $r$ be distinct points at $\m$,
and let $a,b,c$ be the lengths of the sides of the geodesic triangle with vertices
$p,q$ and $r$. Let $\alpha, \beta$ and $\gamma$ denote the angles opposite to the
sides of lengths $a,b$ and $c$ respectively. Then:
\begin{enumerate}
\item $c^{2}\geq a^{2} + b^{2} - 2ab\, cos\gamma$;

\item (Double law of cosines) $c \leq b\, cos\alpha + a \, cos\beta$.
\end{enumerate}

\item \textbf{Angle sum theorem:} The sum of the interior angles of a geodesic triangle
in any simply connected manifold $\m$ with $\kappa \leq 0$ is at most $\pi$.
Actually, this follows from the Gauss-Bonnet Theorem.
\end{itemize}

When $\m$ is a Hadamard surface with sectional curvature bounded above by a negative
constant then any two asymptotic geodesics $\gamma, \beta$ satisfy that the distance
between the two curves $\gamma_{|[t, +\infty)},\beta_{|[t, +\infty)} $ is zero for
any $t \in \r$. For each point $p \in \m$  and $x \in \m(\infty)$, there is an
unique geodesic $\gamma_{px}$ with initial condition $\gamma _{px}(0)=p$ and it is
in the equivalence class of $x$. For each point $p \in \m$ we may identify
$\m(\infty)$ with the circle $\mathbb{S}^{1}$ of unit vectors in $T_{p}\m$ by means
of the bijection
\begin{equation*}
\begin{matrix}
G_p : & \s ^1 \subset T_p \m & \to & \m (\infty) \\
 & v & \longmapsto & \lim _{t\to +\infty} \gamma _{p,v}(t)
\end{matrix}
\end{equation*}where $\gamma _{p,v}$ is the geodesic with initial conditions $\gamma _{p,v}(0)=p$
and $\gamma _{p,v}'(0)=v$. In addition the hypothesis on the sectional curvature (it
is bounded above by a negative constant) yields there is an unique geodesic joining
two points of $\m(\infty)$.

Given a set $\Omega \subseteq \m$, we denote by $\partial_{\infty}\Omega$ the set
$\partial \Omega \cap \m(\infty)$,where $\partial \Omega$ is the boundary of
$\Omega$ for the cone topology. We orient $\m$ so that its boundary at infinity is
oriented counter-clockwise.

Let $\alpha$ be a complete oriented geodesic in $\m$, then
\begin{equation*}
\partial_{\infty}\alpha = \{\alpha^{-}, \alpha^{+}\}
\end{equation*}
where $\alpha^{-} = \lim _{t \to -\infty}\hspace{0,2cm}\alpha(t)$ and $\alpha^{+} =
\lim _{t \to + \infty}\hspace{0,2cm}\alpha(t)$. Here $t$ is arc length along
$\alpha$. We identify $\alpha$ with its boundary at infinity, writing $\alpha =
\{\alpha^{-}, \alpha^{+}\}$.

\begin{definition}\label{Def:orientedgeod}
Let $\theta_{1}$ and $\theta_{2}$ $\in \m(\infty)$, we define the oriented geodesic
joining $\theta_{1}$ and $\theta_{2}$, $\alpha (\theta_{1},\theta_{2})$, as the
oriented geodesic from $\theta_{1} \in \m(\infty)$ to  $\theta_{2} \in \m(\infty)$.
\end{definition}

\begin{definition}\label{Def:interiordomain}
Let $\alpha$ a oriented complete geodesic in $\m$. Let $J$ be the standard
counter-clockwise rotation operator. We call exterior set of $\alpha$ in $\m$,
$ext_{\mathbb{M}^{2}}(\alpha)$, the connected component of $\m \setminus \alpha$
towards which $J\alpha'$ points. The other connected component of $\m \setminus
\alpha$ is called the interior set of $\alpha$ in $\m$ and denoted by
$int_{\mathbb{M}^{2}}(\alpha)$.
\end{definition}

Now, we establish a Lemma that will be used later.

\begin{lemma}\label{Lem:basicfoliation}
Let $\m $ be a Hadamard surface. Let $p \in \m  \setminus \alpha$, where $\alpha$ is
a complete geodesic in $\m $ and $q\in \alpha$ such that $d(p,q)= d(p,\alpha)$. Let
$\beta$ be a complete geodesic joining $p$ to $q$, then $\beta$ intersects
orthogonally $\alpha$ in exactly one point. Here, $d$ denotes the distance function.
\end{lemma}
\begin{proof}
Let $r \in \alpha$ and let $\gamma _{pr}$ be the geodesic joining $p$ to $r$.
Consider the geodesic triangle of vertices $p,q,r$. Set $\varphi := \measuredangle
_q(p,r) $, $\theta := \measuredangle _p(q,r)$ and $\phi :=\measuredangle _r(p,q)$
with lengths of opposite sides $a,b,c$ respectively.

Suppose that $\phi \neq \frac{\pi}{2}$ and suppose, for example, $\phi <
\frac{\pi}{2}$. On the one hand, from the double law of cosines, we have inequality:

\begin{equation*}
c \leq a \, cos\theta + b \, cos\varphi
\end{equation*}

Moving $r$ towards $p$, we get that $\theta \rightarrow 0$ and $\varphi \rightarrow
\pi-\phi$ as $r \rightarrow q$. Therefore we conclude using the above inequality
that

\begin{equation*}
c \leq a + b \, m < a , \text{ where } \,  m = cos \varphi < 0
\end{equation*}but is a contradiction, since $ c := d(p,q)$ is the minimal distance.
\end{proof}

Our next step is to use the above Lemma for proving the following,

\begin{lemma}\label{Lem:folorthogonal}
Let $\m $ be a Hadamard surface with Gaussian curvature $\kappa \leq 0$ and $\alpha$
a complete geodesic in $\m$. Let $s$ be the arc length parameter along $\alpha$. Set
$S = \bigcup _{s\in \r}\beta _s$, where $\beta$ is the complete geodesic in $\m$
orthogonal to $\alpha (s)$ and $\beta _{s}(0)=\alpha (s)$ for all $s\in r$. Then,
$S$ is a foliation of $\m$.
\end{lemma}
\begin{proof}
First of all, from Lemma \ref{Lem:basicfoliation}, we have that $\m \subseteq S$.
So, we only have to prove that $\beta _{s_0} \cap \beta_{s_1}=\emptyset$ for $s_0
\neq s_1$. Actually, this follows from the Gauss-Bonnet formula.

Assume there exist $p \in \beta _{s_0} \cap \beta_{s_1}$. First, note that $p \not
\in \alpha$ since $\beta (s_0) \neq \beta (s_1)$. Now, consider the geodesic
triangle of vertices $p$, $\beta _{s_0}(0)$ and $\beta _{s_1} (0)$. Since $\beta
_{s_0}$ and $\beta _{s_1}$ meet orthogonally to $\alpha$, then the angles $\varphi $
and $\phi$ subtended at $\beta_{s_0}(0)$ and $\beta _{s_1}(0)$ are equal to $\pi/2$
respectively. Moreover, the angle $\theta$ subtended at $p$ is positive. Thus
$\varphi + \phi + \theta > \pi$, which contradicts the Angle Sum Theorem.
\end{proof}

Next, we establish another result about foliations on Hadamard surfaces.

\begin{lemma}\label{Lem:folinfinity}
Let $\m $ be a Hadamard surface with Gaussian curvature $\kappa \leq c < 0$. Let $S
= \bigcup_{y} \alpha(x_{0},y)$, where $x_{0}$ is a fixed point of
$\partial_{\infty}\m$ and $y \in \partial_{\infty}\m\setminus \{x_{0}\}$. Then $S$
is a foliation of $\m$.

\end{lemma}
\begin{proof}
It is clear that $\m \subseteq S$, so we only need to prove that $\alpha(x_{0},y_1)
\cap \alpha(x_{0}, y_2) = \emptyset$ for $y_1 \neq y_2$.

Assume that $r \in \alpha(x_{0}, y_1) \cap \alpha(x,y_2)$. In this case, we have two
distinct geodesics arcs, $\alpha(x_{0}, y_1)$ and $\alpha(x_{0}, y_2)$ joining $r$
to $x_{0}$, a contradiction.
\end{proof}

\subsection{On Killing submersions}

Now, we establish the definitions and properties of a Riemannian $3-$manifold which
fiber over a Riemannian surface and whose fibers are the trajectories of a unit
Killing vector field.

Let $\man$ be a $3-$dimensional Riemannian manifold so that it is a Riemannian
submersion $\pi : \man \to \m$ over a surface $(\m , g)$ with Gauss curvature
$\kappa$, and the {\it fibers}, i.e. the inverse image of a point at $\m $ by $\pi$,
are the trajectories of a unit Killing vector field $\xi $, and hence geodesics.
Denote by $\meta{}{}$, $\camb$, $\ext $, $\bar R$ and $[,]$ the metric, Levi-Civita
connection, exterior product, Riemann curvature tensor and Lie bracket in $\man$,
respectively. Moreover, associated to $\xi$, we consider the operator $J: \campo
(\man) \to \campo (\man)$ given by
\begin{equation*}
J X : = X \ext \xi , \, \, \, X \in \campo (\man).
\end{equation*}

Given $X \in \campo (\man)$, $X$ is {\it vertical} if it is always tangent to
fibers, and {\it horizontal} if always orthogonal to fibers. Moreover, if $X \in
\campo (\man) $, we denote by $X^v$ and $X^h$ the projections onto the subspaces of
vertical and horizontal vectors respectively.

Now, we remind the definition of two tensors that appear naturally (see \cite{O} for
details) when we have a submersion. Given $X, Y \in \campo (\man) $ we define
\begin{equation}\label{TensorT}
\mathcal{T}_X Y = \left( \camb _{X^v} Y^v\right)^h  + \left( \camb _{X^v}
Y^h\right)^v ,
\end{equation}and
\begin{equation}\label{TensorA}
\mathcal{A}_X Y = \left( \camb _{X^h} Y^h\right)^v  + \left( \camb _{X^h}
Y^v\right)^h .
\end{equation}

We will not recall the properties of these tensors, we refer the reader to \cite{O}
for the properties we will make use.

First of all, we will see how we can associate a function to the ambient manifold
$\man $.

\begin{proposition}\label{Prop:tau}
Let $\man$ be as above. There exists a function $\tau : \man \to \r $ so that
\begin{equation}
\camb _X \xi = \tau \, X \ext \xi , \,
\end{equation}
\end{proposition}
\begin{proof}
Set $X\in \campo (\man)$. Since $\xi $ is a unit Killing field, we have
$$  \meta{\camb _{X } \xi }{X } = 0 ,$$and
$$ \meta{\camb _{X} \xi}{\xi } = \frac{1}{2}X\meta{\xi}{\xi} = 0.  $$

Thus, for any horizontal $X \in \campo (\man)$, there exist $\tau _X : \man \to \r $
so that
$$ \camb _{X}\xi = \tau _X \, X \ext \xi .$$

Hence, let $\set{X, Y} \in \campo (\man )$ be an orthonormal basis of horizontal
vector fields so that ${\rm det}(X , Y , \xi) =1$, we have
\begin{eqnarray}
\camb _{X}\xi &=&\tau _X \, X \ext \xi \label{tauX}\\[3mm]
\camb _Y \xi &=& \tau _Y \, Y \ext \xi \label{tauY} .
\end{eqnarray}

Hence, it is enough to prove that $\tau _ X = \tau _Y$. Take the scalar product of
\eqref{tauX} and $Y$; and the scalar product of \eqref{tauY} and $X$, then
\begin{equation*}
\begin{split}
\meta{\camb _X \xi }{Y} &= \tau _X \, \meta{X \ext \xi}{Y} = \tau _X \, {\rm
det}(X,\xi, Y) \\
 &= - \tau _X \, {\rm det}(X,Y,\xi) = - \tau _X,\\
\meta{\camb _Y \xi }{X} &= \tau _Y \, \meta{Y \ext \xi}{X} = \tau _Y \, {\rm
det}(Y,\xi, X) \\
 &=\tau _Y \, {\rm det}(X,Y,\xi) = \tau _Y .
\end{split}
\end{equation*}

Since $\xi$ is a Killing vector field, the above equation yields
\begin{equation*}
0 = \meta{\camb _X \xi}{Y} + \meta{\camb _Y \xi }{X} = -\tau _X + \tau _Y ,
\end{equation*}thus $\tau _X = \tau _Y$.
\end{proof}

Proposition \ref{Prop:tau} makes natural to introduce the following definition:

\begin{definition}
A Riemannian submersion over a surface $\m $ whose fibers are the trajectories of a
unit Killing vector field $\xi$ will be called {\it Killing submersion} and denoted
by $\hm $, where $\kappa $ is the Gauss curvature of $\m $ and $\tau $ is given in
Proposition \ref{Prop:tau}.
\end{definition}

Our first task is to compute the sectional curvature $\bar K(X,Y)$ of any plane
generated by $X,Y \in \campo (\hm)$.

\begin{lemma}\label{Lem:SectK}
Let $\hm $ be a Riemannian submersion with unit Killing vector field $\xi $. Let
$\set{X, Y} \in T \hm $ be an orthonormal basis of horizontal vector fields so that
$\set{X,Y,\xi}$ is positively oriented. Then
\begin{eqnarray}
\bar K( X,Y ) &=& \kappa - 3\tau ^2 , \label{SectHH}\\
\bar K( X,\xi ) &=& \tau ^2 . \label{SectHV}
\end{eqnarray}
\end{lemma}
\begin{proof}
From \cite[Corollary 1]{O}, we have
$$ \bar K(X, Y) = \kappa - 3 \norm{\mathcal{A}_X Y}^2 ,$$and using \cite[Lemma 2]{O} we know
that $\mathcal{A}_X Y = \frac{1}{2}[X,Y]^v$. Thus,
\begin{equation*}
\begin{split}
\meta{\mathcal{A}_X Y}{\xi} &=
\frac{1}{2}\meta{[X,Y]^v}{\xi}=\frac{1}{2}\meta{[X,Y]}{\xi} =
\frac{1}{2}\meta{\camb _X Y}{\xi} - \frac{1}{2}\meta{\camb _Y X}{\xi} = \\
&= -\frac{1}{2}\meta{ Y}{\camb _X \xi} + \frac{1}{2}\meta{ X}{\camb _Y \xi} =
-\frac{1}{2}\meta{ Y}{-\tau Y} + \frac{1}{2}\meta{ X}{\tau X} = \\
&= \tau
\end{split}
\end{equation*}where we have used that $\set{X,Y,\xi}$ is positively oriented, i.e.,
$ \camb _X \xi = - \tau Y$ and $\camb _Y \xi = \tau X$. So,
\begin{equation*}
\mathcal{A} _X Y = \tau \xi ,
\end{equation*}since $\mathcal{A}_X Y$ is vertical. Hence, we obtain \eqref{SectHH}.

Again, \cite[Corollary 1]{O} gives
\begin{equation*}
\bar K(X, \xi ) = \meta{\left(\camb _X \mathcal{T}\right)_\xi \xi}{X} +
\norm{\mathcal{A}_X \xi}^2 - \norm{\mathcal{T}_\xi X} ^2 .
\end{equation*}

On one hand, $\mathcal{A}_X \xi = \left( \camb _X \xi\right)^h$, i.e., it is a
horizontal vector field. Then,
\begin{equation*}
\begin{split}
\meta{\mathcal{A}_X \xi}{X} &= \meta{\camb _X \xi}{X} = -\tau \meta{Y}{X} =0 ,\\
\meta{\mathcal{A}_X \xi}{Y} &= \meta{\camb _X \xi}{Y} = -\tau \meta{Y}{Y} = -\tau ,
\end{split}
\end{equation*}that is, $\mathcal{A}_X \xi = - \tau Y$, thus $\norm{\mathcal{A}_X \xi }^2= \tau ^2$.

On the other hand, $\mathcal{T}_\xi X = \left( \camb _\xi X\right)^v$ is vertical,
so
\begin{equation*}
\meta{\mathcal{T}_\xi X}{\xi } = \meta{\camb _\xi X}{\xi} = - \meta{X}{\camb _\xi
\xi} = 0,
\end{equation*}which implies $\mathcal{T}_\xi X =0 $, hence $\norm{\mathcal{T}_\xi X}^2 =0$.

Finally, since $\camb _\xi \xi = 0 $ and $\mathcal{T}_Y \xi = 0$,
\begin{equation*}
\begin{split}
\left(\camb _X \mathcal{T}\right)_\xi \xi &= \camb _X \mathcal{T}_\xi \xi -
\mathcal{T}_{\camb _X \xi} \xi - \mathcal{T}_\xi \camb _X \xi \\
 &= \camb _X \left( \camb _\xi \xi\right)^h + \tau \mathcal{T}_Y \xi + \mathcal{T}_\xi \left( \tau
 Y\right)\\
 &= \left( \camb _\xi (\tau Y)\right)^v = \left( \xi (\tau) Y + \tau \camb _\xi
 Y\right)^v\\
 &= \tau \left( \camb _\xi Y \right)^v ,
\end{split}
\end{equation*}we obtain
\begin{equation*}
\meta{\left(\camb _X \mathcal{T}\right)_\xi \xi}{X}=0 .
\end{equation*}

Summarizing, $\norm{\mathcal{A}_X \xi }^2= \tau ^2$, $\norm{\mathcal{T}_\xi X}^2 =0$
and $\meta{\left(\camb _X \mathcal{T}\right)_\xi \xi}{X}=0$, thus \eqref{SectHV}
follows from the expression of $\bar K(X, \xi )$.
\end{proof}

\subsection{On surfaces in Killing submersions}

Let $\Sigma \subset \hm$ be an oriented immersed connected surface. We endow
$\Sigma$ with the induced metric (\emph{First Fundamental Form}),
$\meta{}{}_{|\Sigma}$, in $\hm$, which we still denote by $\meta{}{}$. Denote by
$\nabla $ and $R$ the Levi-Civita connection and the Riemann curvature tensor of
$\Sigma$ respectively, and $S$ the shape operator, i.e., $S X = - \nabla _X N$ for
all $X \in \campo (\Sigma)$ where $N$ is the unit normal vector field along the
surface. Then $II(X,Y) = \meta{SX }{Y}$ is the \emph{Second Fundamental Form} of
$\Sigma$. Moreover, we denote by $J$ the (oriented) rotation of angle $\pi /2$ on
$T\Sigma$.

Set $\nu = \meta{N}{\xi}$ and $T = \xi - \nu N$, i.e., $\nu $ is the normal
component of the vertical field $\xi$, called the \emph{angle function}, and $T$ is
the tangent component of the vertical field.

We now study some particular surfaces in $\hm$. To do so, we will require some
definitions.

\begin{definition}\label{Def:cylinder}
We say that $\Sigma \subset \hm$ is a vertical cylinder over $\alpha$ if \,$\Sigma
:= \pi ^{-1} (\alpha)$, where $\alpha $ is a curve on $(\m  ,g)$. If $\alpha $ is a
geodesic, $\Sigma := \pi ^{-1}(\alpha)$ is called a vertical plane.
\end{definition}

Let us start by studying the geometry of a vertical cylinder:

\begin{proposition}\label{Prop:Cylinder}
Let $\Sigma \subset \hm$ be a vertical cylinder over $\alpha$. Then, the mean,
Gaussian and extrinsic curvature are respectively
\begin{equation*}
H = k_g /2 , \quad K = 0 , \quad  K_e = - \tau ^2 ,
\end{equation*}where $k_g$ is the geodesic curvature of $\alpha $ with respect to
$g$. Moreover, these cylinders are characterized by $\nu \equiv 0$. In particular, a
complete vertical cylinder is isometric to $\r ^2$. Also, when $\tau \equiv 0$, a
vertical plane in $\mr$ is totally geodesic.
\end{proposition}
\begin{proof}
Let us parametrize $\alpha \subset \m  $ by arc-length. Let $\vec t$ and $\vec n$ be
the tangent and normal vector fields along $\alpha$. Denote by $\vec T$ and $N$ the
unique horizontal lifts  to $\hm$. Note that $\set{\vec T, \xi} \in \campo(\Sigma)$
is a orthonormal basis and $N$ is the unit normal vector field along $\Sigma$, in
particular $\nu \equiv 0$. Moreover, it is clear that $\Sigma$ is flat, i.e.,
$K\equiv 0$. We choose $N$ so that $\set{\vec T,N,\xi}$ is positively oriented.

The second fundamental form applied to a pair of vector fields $X,Y \in \campo
(\Sigma)$ is given by $II(X,Y) = \meta{\camb _X Y}{N}$. We want to compute the
second fundamental form of $\Sigma$ in the basis $\set{\vec T,\xi}$. Then,
\begin{equation*}
\begin{split}
\camb _\xi \xi &= 0 , \, \, \mbox{ since it is a unit Killing vector field}. \\[4mm]
\meta{\camb _{\vec T} \xi }{N} &= \tau \meta{\vec{T} \ext \xi}{ N } = \tau \, {\rm
det}(\vec T,\xi ,N)=- \tau. \\[4mm]
\meta{\camb_{\vec T} \vec T}{N}&= g\left(\nabla ^{\tm} _{\vec t} \vec t ,\vec n
\right)= k_g .
\end{split}
\end{equation*}

Thus, if we set $X : = x^1 \xi + x ^2 \vec T$ and $Y:= y^1 \xi + y^2 \vec T$, we
have
$$ II(X,Y) = (x^1 , x^2) \begin{pmatrix}
0 & - \tau \\
-\tau & k_g
\end{pmatrix}\begin{pmatrix}
y^1 \\ y^2
\end{pmatrix}. $$

So, since the mean and extrinsic curvatures are the trace and determinant
respectively of the second form, we obtain the result.
\end{proof}

Henceforth, most of the results are proven under the assumption that $\hm$ fibers
over a strict Hadamard surface $\m$, that is, the Gaussian curvature $\kappa$ of
$\m$ is bounded above by a negative constant. Therefore, we define:

\begin{definition}
We say  that $\hm$ is a strict Hadamard-Killing submersion if it fibers over a
strict Hadamard surface $\m $, i.e., $\m$ has Gaussian curvature $\kappa $ bounded
above by a negative constant.
\end{definition}

We will introduce a definition according to that given for complete geodesics in a
Hadamard surface since the notions of interior and exterior domains of a horizontal
oriented geodesic extend naturally to vertical planes.

\begin{definition}\label{Def:Interiorplane}
Let $\hm$ be a Hadamard-Killing submersion. For a complete oriented geodesic
$\alpha$ in $\m$ we call, respectively, interior and exterior of the vertical plane
$P = \pi^{-1}(\alpha)$ the sets
\begin{equation}
int_{\mathcal{M}(\kappa,\tau)}(P)=\pi^{-1}(int_{\mathbb{M}^{2}}(\alpha)),
\hspace{1cm}
ext_{\mathcal{M}(\kappa,\tau)}(P)=\pi^{-1}(ext_{\mathbb{M}^{2}}(\alpha)) \nonumber
\end{equation}
\end{definition}

Moreover, we will often use foliations by vertical planes of $\hm$. We now make this
precise.

\begin{definition}\label{Def:folitaionplanes}
Let $\hm$ be a Hadamard-Killing submersion. Let $P$ be a vertical plane in $\hm$,
and let $\beta(t)$ be an oriented horizontal geodesic in $\m $, with $t$ arc length
along $\beta$, $\beta(0)=p_{0} \in P$, $\beta'(0)$ orthogonal to $P$ at $p_{0}$ and
$\beta(t) \in ext_{\mathcal{M}(\kappa,\tau)}(P)$ for $t> 0$. We define the oriented
foliation of vertical planes along $\beta$, denoted by $P_{\beta}(t)$, to be the
vertical planes orthogonal to $\beta(t)$ with $P=P_{\beta}(0)$.
\end{definition}

\begin{remark}
The Definition \ref{Def:folitaionplanes} is actually a foliation by Lemma
\ref{Lem:folorthogonal}.
\end{remark}

To finish, we will give the definition of a particular type of curve in a vertical
plane. To do so, we recall a few concepts about Killing graphs in a Killing
submersion (see \cite{RST}).

Under the assumption that the fibers are complete geodesics of infinite length, it
can be shown (see \cite{St}) that such a fibration is topologically trivial.
Moreover, there always exists a global section
$$ s : \m \to \hm ,$$so, considering the flow $\phi _t$ of $\xi$, a trivialization
of the fibration is given by the diffeomorphism
$$ \begin{matrix}
\mr & \to & \hm \\
(p,t)& \rightarrowtail & \phi _t (s(p))
\end{matrix} $$

\begin{definition}\label{Def:Killinggraph}
Let $\pi : \hm \to \m $ be a Killing submersion. Let $\Omega \subset \m $ be a
domain. A Killing graph over $\Omega$ is a surface $\Sigma \subset \hm $ which is
the image of a section $s : \overline \Omega \to \hm $, with $s \in C^2(\Omega) \cap
C^0 (\overline \Omega)$. We may also consider graphs, $\Sigma \subset \hm$, without
boundary.
\end{definition}

Now, we can establish the announced definition.

\begin{definition}\label{Def:veticalcurve}
Let $P$ be a vertical plane in $\hm$ and $\alpha$ a complete embedded convex curve
in $P$. We say that $\alpha$ is an untilted curve in $P$ if there exists a point
$p\in \alpha$ so that $\phi _t (p)$ is contained in the convex body bounded by
$\alpha $ in $P$ for all $t>0$ (or $t<0$). Otherwise, we say that $\alpha $ is
tilted.
\end{definition}

\section{Hadamard-Stoker-Currier type theorems}\label{HSTheorem}

We devote this section to the proof of a Hadamard-Stoker-Currier type theorem in a
strict Hadamard-Killing submersion. First, note that if $\Sigma \subset \hm $ is an
immersed surface with positive extrinsic curvature, then we can choose a globally
defined unit normal vector field $N$ so that the principal curvatures, i.e., the
eigenvalues of the shape operator, are positive. We denote them by $k_i$ for
$i=1,2$.

We start with the following elementary result.

\begin{proposition}\label{Pro:convex}
Let $\Sigma \subset \hm $ be an immersed surface whose principal curvatures satisfy
$k_{i}(p)> |\tau (p)|$ for all $p\in \Sigma$. Let $P$ be a vertical plane. If
$\Sigma$ and $P$ intersect transversally then each connected component $C$ of
$\Sigma \cap P$ is a strictly convex curve in $P$.
\end{proposition}
\begin{proof}
Let us parametrize $C$ as $\gamma(t)$ where $t$ is arc length. Then
\begin{equation*}
\nabla_{\gamma'}^{P}\gamma' + II_{P}(\gamma', \gamma')N_{P} =
\overline{\nabla}_{\gamma'}\gamma'= \nabla_{\gamma'}\gamma'+ II(\gamma', \gamma')N
\end{equation*}
where $\nabla^{P}$, $\overline{\nabla}$ and $\nabla$ are the connections on $P$,
$\hm$ and $\Sigma$ respectively, $II_{P}$ and $II$ are the second fundamental forms
of $P$ and $\Sigma$ respectively, and $N_{P}$ and $N$ are the unit normal vector
fields along $P$ and $\Sigma$, respectively.

Taking inner product in the above equality,
\begin{equation*}
\|\nabla_{\gamma'}^{P}\gamma'\|^{2} + II_{P}(\gamma', \gamma')^{2}=
\|\nabla_{\gamma'}\gamma'\|^{2} + II(\gamma', \gamma')^{2} .
\end{equation*}

Thus,
\begin{equation}
\begin{split}
 \|\nabla_{\gamma'}^{P}\gamma'\|^{2} &= \|\nabla_{\gamma'}\gamma'\|^{2} +
II(\gamma', \gamma')^{2} - II_{P}(\gamma', \gamma')^{2} \\
 & \geq II(\gamma', \gamma')^{2} - II_{P}(\gamma', \gamma')^{2} > 0, \nonumber
\end{split}
\end{equation}
since $k_{i}> |\tau|$.
 Thus $\nabla_{\gamma'}^{P}\gamma' \neq 0$, that is, the
geodesic curvature of $C$ vanishes nowhere on $P$.
\end{proof}

\begin{definition}\label{Def:simpleend}
Let $\hm $ be a strict Hadamard-Killing submersion. Let $\Sigma \subset \hm$ be a
surface. We say that $\Sigma$ has a simple end if the boundary at infinity of
$\pi(\Sigma)$ is a unique point $\theta_{0}\in \mathbb{M}(\infty)$ and, in addition,
for all $\theta_{1}, \theta_{2} \in \mathbb{M}(\infty) \backslash \{\theta_{0}\}$
the intersection of the vertical plane $\pi^{-1}(\alpha)$ and $\Sigma$ is empty or a
compact set, where $\alpha$ is a geodesic joining $\theta_{1}$ to $ \theta_{2}$.
\end{definition}

At this point we have enough information to prove our main result.

\begin{theorem}\label{Theo:HS}
Let $ \Sigma \subset \hm$ be a complete connected immersed surface so that $k_{i}(p)
> |\tau (p)|$ for all $p\in \Sigma$, where $\hm$ is a strict Hadamard-Killing submersion.
Then $\Sigma$ is properly embedded. Moreover, $\Sigma$ is homeomorphic to
$\mathbb{S}^{2}$ or to $\mathbb{R}^{2}$. In the later case, $\Sigma$ has a simple
end or $\Sigma$ is a Killing graph over a convex domain of $\m$.
\end{theorem}
\begin{proof}
The proof follows the ideas in \cite[Theorem 3.1]{EGR}. We distinguish two cases
depending on the existence of a point $p$ at $\Sigma$ so that $N(p)$ is horizontal,
that is, $N(p)$ is orthogonal to the fiber $\xi$.

{\bf Case 1:} {\it Suppose there is no point $p \in \Sigma$ so that $N(p)$ is
horizontal. Then, $\Sigma $ is embedded and homeomorphic to the plane. Moreover, it
is a Killing graph over a convex domain in $\m$.}

{\it Proof of Case 1:} For proving this case, we first show the following

{\bf Assertion 1:} Let $P$ be a vertical plane that meets transversally $\Sigma$,
then each connected component of $\Sigma \cap P$ is an open embedded strictly convex
curve. Moreover, each connected component is a Killing graph over an open interval
in $\alpha$, where $\alpha$ is the complete geodesic in $\m$ so that $P:=\pi
^{-1}(\alpha)$

{\it Proof of the Assertion 1:} We only need to show that each connected component
is embedded, since we already know that it is strictly convex by Proposition
\ref{Pro:convex}.

Let $C$ be a connected component of $\Sigma \cap P$ that is not embedded, then it
has a loop $L \subset C$ homeomorphic to a circle, i.e, there exists a homeomorphism
$c : \s ^1 \to L $ and $c' \neq 0$ except at one point. Clearly, there is a point $q
\in c(\s ^1)$ where $c'$ is vertical. Then $\nu (q)=0$ which is a contradiction.

%
%

This argument also proves that a connected component can not be compact. Also, it
proves that each connected component is a Killing graph over $\alpha$, where
$\alpha$ is the complete geodesic in $\m$ so that $P:=\pi ^{-1}(\alpha)$. This
proves the Assertion 1.

\vspace{.5cm}

Now, let $P$ be a vertical plane which meets $\Sigma$ transversally and
$P_{\beta}(t)$ be the oriented foliation of vertical planes along $\beta$ (see
Definition \ref{Def:folitaionplanes}). From Assertion 1, each connected component is
an open embedded strictly convex curve. Let $C(0)$ be an embedded component of
$P_{\beta }\cap \Sigma$. Let us consider how $C(0)$ varies as $t$ increases to
$+\infty$. First, note that no two components of $P_{\beta}(0) \cap \Sigma$ can join
to the component $C(t_0)$ associated to $C(0)$ at some $t_0 >0$. Otherwise, the unit
normal vector field $N$ should point up in a component and down in the other for
$t_0 -\epsilon < t<t_0$ ($\epsilon$ small enough), since $N$ is globally defined.
Thus, by continuity, this would produce a point where $N$ is horizontal, a
contradiction. Hence, from Assertion 1, the component $C(0) \subset P_{\beta}(0)\cap
\Sigma$ varies continuously to one open embedded curve $C(t) \subset
P_{\beta}(t)\cap \Sigma$ as $t$ increases. The only change possible is that $C(t)$
goes to infinity as $t$ converges to some $t_1$ and disappears in $P_{\beta}(t_1)$.
Similarly $C(0)$ varies continuously to one embedded curve of $P_{\beta}(t) \cap
\Sigma$ as $t\to - \infty$. Hence $\Sigma$ connected yields $P_{\beta}(t) \cap
\Sigma $ is at most one component for all $t$. So, we observe that $P_{\beta}(t)
\cap \Sigma$ is empty or homeomorphic to $\r$ for each $t$, hence $\Sigma $ is
topologically $\r^2$. To finish, again from Assertion 1, we conclude $\Sigma$ is a
Killing graph. The fact that $\Sigma$ is a Killing graph over a convex domain
follows from Proposition \ref{Pro:convex}.

This completes the proof of Case 1.

\vspace{.5cm}

{\bf Case 2:} {\it Suppose there is a point $p_0 \in \Sigma$ so that $N(p_0)$ is
horizontal. Then, $\Sigma $ is embedded and homeomorphic to the sphere or to the
plane, in which case, $\Sigma $ has a simple end.}

By assumption $N$ is horizontal at $p_0$ and so, the tangent plane $T_{p_0} \Sigma$
is spanned by $\set{\xi (p_0) , X(p_0)}$, where $X(p_0)$ is horizontal. Set $\bar
p_0 := \pi (p_0)$ and $v := d\pi _{p_0} (X(p_0))$. Let $\alpha $ be the complete
geodesic in $\m$ with initial conditions $\alpha (0) = \bar p_0$ and $\alpha ' (0) =
v$. Set $P:=\pi ^{-1}(\alpha)$. Note that $p_0 \in P \cap \Sigma$ and the principal
curvatures of $\Sigma$ at $p_0$ are greater than the principal curvatures of $P$ at
$p_0$, thus $\Sigma$ lies (locally around $p_0$) on one side of $P$. Without loss of
generality we can assume that $N(p_0)$ points to $ext_{\mathcal{M}(\kappa,\tau)}(P)$
(see Definition \ref{Def:interiordomain}), therefore, $\Sigma$ lies (locally around
$p_0$) in $ext_{\mathcal{M}(\kappa,\tau)}(P)$. Moreover, we parametrize the boundary
at infinity by $B: [0,2\pi] \to \m (\infty)$ so that $B(0)= \alpha ^-$, $B(\pi)=
\alpha ^+$ and $\partial _{\infty} ext_{\mathcal{M}(\kappa,\tau)}(P) = B([0,\pi])$.
Also, from now on, we identify the points at infinity with the points of the
interval $[0,2\pi]$.

Let $N_P$ be the unit normal vector field along $P$ pointing into
$ext_{\mathcal{M}(\kappa,\tau)}(P)$. Then, there exists neighborhoods $V \subset P$
and $U \subset \Sigma$ so that
$$ U := \set{ {\rm exp}_q (f(q)N_P(q)) : \, q \in V } ,$$where
$f:V \to \r $ is a smooth function and ${\rm exp}$ is the exponential map in $\hm$.

Let $P_{\beta}(t)$ be the foliation of vertical planes along $\beta$ (see Definition
\ref{Def:folitaionplanes}). From Proposition \ref{Pro:convex} and the fact that
locally $\Sigma$ is (in exponential coordinates) a graph, there is  $\epsilon > 0$
such that the curves $P_{\beta}(t) \cap U$ are embedded strictly convex curves (in
$P_{\beta}(t)$) for $0<t<\epsilon$. Perhaps, $P_{\beta}(t) \cap \Sigma$ has other
components distinct from $C(t)$ for each $0<t < \epsilon$, but we only care how
$C(t)$ varies as $t$ increases. We also denote by $C(t)$ the continuous variation of
the curves $P_{\beta}(t)\cap \Sigma$ when $t> \epsilon$.

We distinguish two cases:

{\bf Case A:} {\it If $C(t)$ remains compact as $t$ increases, then $\Sigma$ is
properly embedded and homeomorphic to the sphere or to the plane. In the later case,
$\Sigma$ has a simple end}

{\it Proof of Case A:} By topological arguments, if $C(t)$ remains compact and
non-empty as $t$ increases, then the $C(t)$ remains embedded. So, $C(t)$ is either
embedded compact strictly convex curves for all $t$, or embedded compact strictly
convex curves until $\bar t$ and this component either it becomes a point, or it
drifts off to infinity.

If $C(t)$ remains compact and non-empty as $t \rightarrow +\infty$, then since
$\Sigma$ is connected, $\Sigma$ must be embedded. In addition, because $C(0)$ is a
point and $C(t)$ is homeomorphic to a circle for every positive $t$, $\Sigma$ is
homeomorphic to $\r^2$.

Now, from the fact that $C(t)$ remains compact, then
$$\partial _{\infty} \pi(\Sigma) = \{B(\theta_0)\},$$where $B(\theta _0) = \beta ^+
$, and $\Sigma $ has a simple end.

If there exists $\bar{t} > 0 $ such that $C(t)$ are compact for all $0 < t <
\bar{t}$ and the component $C(t)$ disappears for $t > \bar{t}$, then, $\Sigma$
connected yields that it is either compact, embedded and topologically $\s ^2$ or
non compact, embedded and topologically $\r^2$. That is, if the $C(t)$ converge to a
non empty compact set as $t$ converges to $\bar{t}$ then $C(\bar{t})$ must be a
point (because our surface has no boundary) and $\Sigma$ is a sphere. Otherwise the
$C(t)$ drift off to infinity as $t$ converges to $\bar{t}$ and $\Sigma$ is
topologically a plane.

We now show that in the latter case, the vertical projection $\pi$ of $\Sigma$ has
asymptotic boundary one of the two points at infinity of $\pi(P_{\beta}(\bar{t}))$.

Without lost of generality we can assume that $P_{\beta}(\bar{t})= \pi
^{-1}\left(\gamma \right)$, $\gamma = \set{\gamma ^{-}, \gamma ^{+}}$ where
$B(\theta ^-) = \gamma ^{-}$ and $B(\theta ^{+}) = \gamma ^+$.  Note that $\theta ^-
\in (0,\theta _0)$ and $\theta ^+ \in (\theta _0 , \pi)$. Consider the vertical
plane $Q= \pi ^{-1}(\beta )$. Let $\tilde{C}$ be the component of $Q \cap \Sigma$
containing $p_0$. First observe that $\tilde{C}$ is compact, otherwise it would
intersect $\pi ^{-1}(r)$, where $r:= \pi \left(Q \right)\cap \pi \left(
P_{\beta}(\bar{t})\right)\in \m$, in two points, which is not the case. Thus, we can
consider the disk $\tilde{D}$ bounded by $\tilde{C}$ on $\Sigma$.

Let $Q_{\gamma}(t)$ denote the foliation by vertical planes along $\gamma$,
$Q_{\gamma}(0)=Q$. There exists $t_0$ (we can assume $t_0<0$) satisfying
$Q_{\gamma}(t_0)$ touches $\tilde{D}$ on one side of $\tilde{D}$ by compactness. Let
$q_0\in \tilde{D}\cap Q_{\gamma}(t_0)$ be the point where they touch. Consider the
variation $\tilde{C}(t)$ of $q_0$ on $\Sigma \cap Q_{\gamma }(t)$ from $t=t_0$ to
infinity. Then, $\tilde{C}(t)$ is a convex embedded curve for $t$ in a maximal
interval $ (t_0,\bar{t}_0)$ with $0<\bar{t}_0\leq\infty$. Hence, $\Sigma$ is
foliated by the $\tilde{C}(t)$, $\tilde{C}=\tilde{C}(0)=Q \cap \Sigma$ and $\theta
^-\not\in\partial_{\infty}\pi(\Sigma)$ because $\Sigma$ is on one side of
$Q_{\gamma}(t_0)$.

Now, we will show that $\partial_{\infty}\pi (\Sigma )=\set{B(\theta ^+)}$. Let
$\gamma (\theta):= \gamma (\theta ^*,\theta) $ where $B(\theta ^*)=\beta ^- $ (see
Definition \ref{Def:orientedgeod}), for $\theta \in [0 ,\pi]$. Let $\bar{\theta}$ be
the value of $\theta$ such that $\gamma (\bar{\theta})$ is asymptotic to $\gamma
^{+}$. Let $Q(\theta) = \pi ^{-1}\left(\gamma (\theta )\right)$. For each $\theta$,
$\bar{\theta} < \theta \leq \pi$, we have $\Sigma \cap Q(\theta)$ is one connected
embedded compact curve $C'(\theta)$. The proof of this is the same as the previous
one for $\tilde{C} $. Notice that each $C'(\theta)$ is non empty since $p_0 \in
C'(\theta)$.

Now $C'(\bar{\theta})$ can not be compact, otherwise $\Sigma $ could not be
asymptotic to the plane $P_{\beta} (\bar{t})$, a contradiction.

In order to complete the proof, we show that $\Sigma$ has a simple end. Observe that
$C'(\theta)$ is compact, $\bar{\theta}<\theta<\theta _0$ because $\Sigma
=\cup_{0\leq t<\bar{t}} C(t)$. Moreover, $C'(\theta)\subset \tilde{D}$,
$0<\theta<\theta _0$, and $ \tilde{D}$ is compact. Thus, it is easy to conclude that
$\Sigma$ has a simple end.

Thus we have proved that $\Sigma $ is either a properly embedded sphere or $\Sigma $
is a properly embedded plane with a simple end. This proves Case A.

\vspace{.5cm}

{\bf Case B:} {\it If $C(t)$ becomes non-compact, then $\Sigma$ is a properly
embedded plane with a simple end.}

{\it Proof of Case B:} Let $\bar{t} >0$ be the smallest $t$ with $C(\bar{t})$
non-compact, $C(\bar{t})$ the limit of the $C(t) $ as $t \rightarrow \bar{t}$,
$C(\bar{t})$ is an embedded strictly convex curve in $P_{\beta}(\bar{t})$.

\vspace{.2cm}

\textbf{Claim 1:} {\it $C(\bar{t})$ is tilted (see Definition
\ref{Def:veticalcurve})}.

{\it Proof of Claim 1:} Let us assume that $C(\bar{t})$ is untilted, then there is
point $q \in C(\bar{t})$ so that $\pi ^{-1} (\bar q)$ touches once to $C(\bar t)$,
where $\bar q :=\pi(q)$. First of all, note that $ \tilde{\Sigma} = \bigcup _{0\leq
t \leq \bar{t}} C(t) \subset \Sigma $ is embedded. Let $\Gamma _{\bar p \bar q }$ be
the complete horizontal geodesic (in $\m$) joining $\bar p$ and $\bar q$. Let $ Q =
\pi ^{-1}\left(\Gamma _{\bar p \bar q}\right)$, and consider $\pi^{-1}(r_0)$ where
$r_0 :=\pi \left( Q \right) \cap \pi \left( P_{\beta} (0)\right)$ and $\pi
^{-1}(r_{\bar{t}})$ where $r_{\bar t}:=\pi \left( Q \right)\cap\pi \left(
P_{\beta}(\bar{t})\right)$. Note that $\pi^{-1}(r_0)$ and $\pi^{-1}(r_{\bar{t}})$
are parallel lines in $Q$. Also, $\alpha _Q = Q \cap \tilde{\Sigma}$ is a
non-compact embedded strictly convex curve in $Q$ such that $\pi ^{-1}(r_0)$ is
tangent to $\alpha _Q$ at $p_0 \in \alpha _Q $ and $\alpha _Q \cap \pi
^{-1}(r_{\bar{t}})$ is exactly one point, since $C(\bar{t})$ is untilted. But this
is a contradiction because $\alpha _Q$ is a strictly convex curve in $Q$, which is
isometric to $\r ^2$, and it must intersect $\pi ^{-1}(r_{\bar{t}})$ twice. Thus,
$C(\bar{t})$ is tilted.

\vspace{.2cm}

And we claim that

\vspace{.2cm}

\textbf{Claim 2:} {\it $\bi \pi (C(\bar{t}))$ is one point.}

{\it Proof of Claim 2:} Let us denote by $D(t)$ the convex body bounded by $C(t)$ in
$P_{\beta}(t)$ for each $0 < t < \bar{t}$. Thus, the limit, $D(\bar{t})$, of $D(t)$
as $t$ increases to $\bar{t}$ is an open convex body bounded by $C(\bar{t})$ in
$P_{\beta}(\bar{t})$ which is isometrically $\r ^2$. If $\bi \pi (C(\bar{t}))$ has
two points, the only possibility is that $C(\bar{t})$ is untilted, which is
impossible by Claim 1.

\vspace{.2cm}

Set $P_{\beta}(\bar{t})= \pi ^{-1}(\gamma )$, $\gamma = \set{\gamma ^{-}, \gamma
^{+}}$ where $B(\theta ^-) = \gamma ^{-}$ and $B(\theta ^{+}) = \gamma ^+$. Note
that $\theta ^- \in (0,\theta _0)$ and $\theta ^+ \in (\theta _0 , \pi)$. From Claim
2, we may assume that $\bi \pi (C(\bar{t})) = \set{B(\theta ^-)}$.

Let $\delta _0> 0 $ and $t_{\delta _0} < \bar{t}$ such that $P_{\beta}(t_{\delta
_0}) = \pi ^{-1}\left(\Gamma (\delta _0) \right)$ where $\Gamma (\delta _0) :=
\set{B(\theta ^- -\delta _0), B(\theta ^+ + \delta _0)}$ (we may assume this by
choosing $B$ in the right way). We denote by $\tilde{\Sigma}_1 =\bigcup _{0\leq t
\leq t_{\delta _0}} C(t)\subset \Sigma $ and note that $\tilde{\Sigma}_1$ is
connected and embedded.

Let us consider the complete horizontal geodesic given by $\Gamma (\delta _0
,s):=\set{B(\theta ^- - \delta _0 + s), B(\theta ^+ +\delta _0)}$ and the vertical
plane $Q(s)=\pi^{-1}\left(\Gamma (\delta _0,s)\right)$ for each $0 \leq s \leq
\theta ^+ -\theta ^-$ (note that $Q(s)$, for $0<s<\theta^+- \theta ^-$ is a
foliation of $ext_{\mathcal{M}(\kappa,\tau)}(Q(0))$). So, $Q(0) =
P_{\gamma}(t_{\delta _0})$ and $Q(0) \cap \tilde{\Sigma}_1 = C(t_{\delta _0})$ is an
embedded compact strictly convex curve. Let us consider how $\alpha (s) = Q(s) \cap
\Sigma $ varies as $s$ increases to $\theta ^+ - \theta ^- -\delta _0$. At this
point, we have two cases:
\begin{enumerate}
\item {\it If $\alpha (s) $ remains compact for all $ 0 \leq s < \theta ^+ - \theta ^- -\delta
_0$, then $\Sigma $ is properly embedded, homeomorphic to the plane and has a simple
end.}

In this case, letting $\delta _0\rightarrow 0$, falls into Case A. So, it is easy to
realize that $\Sigma $ is properly embedded, homeomorphic to the plane and has a
simple end at $B(\theta ^+) \in \m (\infty)$.

\item  {\it $\alpha (s) $ can not become non-compact.}

Let us assume that $\alpha (s)$ becomes non-compact. Let $0 <\bar{s} <\theta ^+ -
\theta ^- -\delta _0$ be the smallest $s$ with $\alpha(\bar{s})$ non-compact,
$\alpha (\bar{s})$ is the limit of the $\alpha(s) $ as $s \rightarrow \bar{s}$.
Also,
\begin{equation*}
\bi \pi (\alpha (\bar{s})) = \{ B(\theta ^- -\delta_0 + \bar{s}) \},
\end{equation*}
otherwise it must be $\set{B(\theta ^+ +\delta _0)}$ which contradicts that
$C(t_{\delta _0})$ is compact.

Clearly $\delta_0<\bar{s}$. For each $\delta \leq \delta _0$ we consider the
complete horizontal geodesic given by $ \sigma (\delta) = \set{ B(\theta ^-
-\delta_0 + \bar{s} -\delta) , B(\theta ^+ +\delta)} $ and the vertical plane
$T(\delta) = \pi ^{-1}\left(\sigma (\delta) \right)$. Let us denote by
$\tilde{\Sigma} _2=\bigcup _{0\leq s \leq \bar{s}-2\delta _0} \alpha (s) \subset
\Sigma $ and note that $\tilde{\Sigma }_2$ is connected and embedded, so,
$\tilde{\Sigma }=\tilde{\Sigma}_1 \cup \tilde{\Sigma}_2 \subset \Sigma$ is connected
and embedded. For each $\delta$, $ 0 < \delta \leq \delta _0 $, $E ( \delta ) = T(
\delta ) \cap \tilde{\Sigma}$ is a strictly convex compact embedded curve in
$T(\delta)$. As $\delta \rightarrow 0$, theses curves converge to a convex curve in
$T( 0 )$ with $\partial _{\infty} \pi ( E( 0 ) )$ the two points $\set{ B(\theta ^-
-\delta_0 + \bar{s}) , B(\theta ^+ )}$. This contradicts Claim 2. Hence $\alpha (s
)$ can not become non-compact.
\end{enumerate}

This proves Claim B, and so Theorem \ref{Theo:HS}.
\end{proof}

In particular, we can recover (and generalize) the afore mentioned \cite[Theorem
3.1]{EGR}.

\begin{corollary}\label{Cor:prodnegative}
Let $\Sigma$ be a complete connected immersed surface in $\m  \times \r$ with
positive extrinsic curvature, where $\mathbb{M}^{2}$ is a Hadamard surface with
Gaussian curvature bounded above by a negative constant. Then $\Sigma$ is properly
embedded and bounds a strictly convex domain in $\m \times \r$. Moreover, $\Sigma$
is homeomorphic either to $\mathbb{S}^{2}$ or to $\mathbb{R}^{2}$. In the later
case, $\Sigma $ is either a graph over a convex domain in $\m$ or $\Sigma$ has a
simple end.
\end{corollary}

In a product space $\tau =0$, and so, since the extrinsic curvature is the product
of the principal curvatures, it is enough to ask that the extrinsic curvature is
positive. Moreover, the assertion about that $\Sigma$ bounds a strictly convex
domain follows from the fact that, in a product space, the vertical planes are
totally geodesics and Proposition \ref{Pro:convex}.

\begin{remark}
Corollary \ref{Cor:prodnegative} is sharp in the sense that there exists complete
embedded surfaces with positive extrinsic curvature and a simple end in $\h ^2
\times \r$ (see \cite[Section 4]{EGR}). Moreover, vertical cylinder in a product
space has zero extrinsic curvature.
\end{remark}

\begin{remark}
Moreover, Theorem \ref{Theo:HS} can be applied to surfaces in $\widetilde{{\rm
PSL}(2,\r)}$ whose principal curvatures are greater than the curvature of the fiber
$\tau$. It would be interesting to investigate the existence of examples in
$\widetilde{{\rm PSL}(2,\r)}$ with principal curvatures greater that the curvature
of the fiber $\tau$ and a simple end.
\end{remark}

\end{document}